\documentclass[11pt, a4paper]{amsart}
\usepackage{amssymb, array, amsmath, amscd, pdfpages, enumerate, amsthm, fixltx2e, setspace, hyperref, bbm,overpic}

\DeclareMathOperator*{\R}{Re}
\DeclareMathOperator*{\I}{Im}

\newcommand{\MK}{M_{{K}}}
\newcommand{\dd}{\mathrm{d}}

\newcommand{\RR}{\mathbb{R}}
\newcommand{\CC}{\mathbb{C}}

\newcommand{\NN}{\mathbb{N}}
\newcommand{\ep}{\varepsilon}
\newcommand{\inv}{^{-1}}
\newcommand{\T}{(T(t))_{t\ge 0}}

\newtheorem{thm}{Theorem}[section]

\newtheorem{lem}[thm]{Lemma}

\theoremstyle{definition}
\newtheorem{rem}[thm]{Remark}

\numberwithin{equation}{section}

\begin{document}
\title[Optimality of the quantified Ingham-Karamata theorem]{Optimality of the quantified Ingham-Karamata theorem for operator semigroups with general resolvent growth}

\author[G. Debruyne]{Gregory Debruyne}
\thanks{G.D., Postdoctoral Fellow of the Research-Council Flanders, gratefully acknowledges support by the FWO through the grant number 12X9719N}
\address[G. Debruyne]{Department of Mathematics\\ Ghent University\\ Krijgslaan 281\\ B 9000 Ghent\\ Belgium}
\email{gregory.debruyne@ugent.be}
\author[D. Seifert]{David Seifert}
\address[D. Seifert]{St John's College \\ St Giles \\ Oxford OX1 3JP\\ United Kingdom}
\email{david.seifert@sjc.ox.ac.uk}

\begin{abstract}
We prove that a general version of the quantified Ingham-Karamata theorem for $C_0$-semigroups is sharp under  mild conditions on the resolvent growth, thus generalising the results contained in a recent paper by the same authors. It follows in particular that the well-known Batty-Duyckaerts theorem is optimal even for bounded $C_0$-semigroups whose generator has subpolynomial resolvent growth. Our proof is based on an elegant application of the open mapping theorem, which we complement by a crucial technical lemma allowing us to strengthen our earlier results.

\end{abstract}


\maketitle

\section{Introduction}\label{sec:intro}

Quantified Tauberian theorems have many important applications in areas ranging from number theory to partial differential equations, but they are also of considerable intrinsic interest. Over the past decade there has accordingly been a great deal of work exploring quantified Tauberian theorems, their applications and their optimality. Of particular interest in many cases, especially those dealing with or at least motivated by applications to energy decay in damped wave equations, are quantified Tauberian theorems for operator semigroups; see for instance \cite{BaBoTo16, BaChTo16,BD08, BT10,ChiSei16, Sta18}. The following result, which is proved in \cite{Sta18}, is a quantified version of the classical Ingham-Karamata theorem \cite{Ing33, Ka34} for $C_0$-semigroups; we refer the reader to \cite[Chapter~III]{Kor04} for a historical overview of the Ingham-Karamata theorem and to \cite{DebVind1, DebVind2} for some important recent contributions on the unquantified version of the result. Throughout this paper, given a continuous non-decreasing function $M\colon\RR_+\to(0,\infty)$ we define the region $\Omega_M$ by
\begin{equation}\label{eq:Omega}
\Omega_M=\left\{\lambda\in\CC:\R\lambda>-\frac{1}{M(|\I\lambda|)}\right\}.
\end{equation}
\begin{thm}\label{thm:MK}
Let $X$ be a complex Banach space and let $A$ be the generator of a bounded $C_0$-semigroup $\T$ on $X$. Suppose that $M,K\colon\RR_+\to(0,\infty)$ are continuous non-decreasing functions such that for some $\delta>1$ the region $\Omega_{\delta M}$
is contained in the resolvent set of $A$ and 
\begin{equation}\label{eq:K_bound}
\sup_{\lambda\in\Omega_{\delta M}}\frac{\|R(\lambda,A)\|}{K(|{\I\lambda}|)}<\infty.
\end{equation}
 Suppose further that there exists $\ep\in(0,1)$ such that
\begin{equation}\label{eq:K_bound2}
K(s)=O\big(\exp\big(\exp\big((sM(s))^{1-\ep}\big)\big)\big),\quad s\to\infty.
\end{equation}
Then there exists a constant $c\in(0,1)$ such that 
\begin{equation}\label{eq:MK_rate}
\|T(t)A\inv\|=O\big(\MK\inv(ct)\inv\big),\quad t\to\infty,
\end{equation}
where $\MK\colon\RR_+\to(0,\infty)$ is the function defined by $\MK(s)=M(s)(\log(1+s)+\log(1+K(s)))$, $s\ge0$.
\end{thm}

In the important special case where $M=K$ the result was first proved in \cite{BD08}. Note that in this case it suffices to assume that $\sigma(A)\cap i\RR=\emptyset$ and that $\|R(is,A)\|=O(M(|s|))$ as $|s|\to\infty$, since this already implies, by a standard Neumann series argument, that $\Omega_{\delta M}$ is contained in the resolvent set $\rho(A)$ of $A$ for all $\delta>1$ and that the appropriate version of \eqref{eq:K_bound} holds. Moreover, the estimate \eqref{eq:K_bound2} is trivially satisfied in this case. If we assume that $\|R(is,A)\|\le M(|s|)$, $s\in\RR$, then the appropriate version of \eqref{eq:MK_rate} with $K=M$ in fact holds for \emph{all} $c\in(0,1)$, as is shown in \cite{Sta18}. What this discussion of the case $M=K$ shows is that the added generality of Theorem~\ref{thm:MK} is of value only when $K(s)\ge M(s)$, $s\ge0$. Remaining for the moment in the case where $M=K$, it is clear that the sharpest rate in \eqref{eq:MK_rate} is then obtained by choosing  $M(s)=\smash{\sup_{|r|\le s}}\|R(ir,A)\|,$ $s\ge0,$
 and for this choice it is shown in \cite{BD08} that one always has the lower bound $\|T(t)A\inv\|\ge  CM\inv(ct)\inv$ for some constants $C,c>0$ and all sufficiently large $t>0$; see also \cite[Section~4.4]{ABHN11}. Here $M\inv$ denotes any right-inverse of the function $M$. This raises the question whether the upper bound in \eqref{eq:MK_rate} is sharp. We remark that if $M_{K}$ grows at least polynomially then the precise value of $c$ in \eqref{eq:MK_rate} is insignificant as the constant $c$ can then generally be absorbed in the $O$-constant. However, this is no longer the case when $M_{K}$ is of subpolynomial growth, in which case the precise value of the constant $c$ can be crucial. If  $M=K$ and both functions have fairly rapid growth, for instance if both are exponential functions, then $M\inv$ and $\smash{\MK\inv}$ have the same asymptotic behaviour. Hence at least in this important special case the optimality question is of limited interest and we will incur no great loss of generality by assuming, as we do in our main results later on, that $M_K$ grows at most exponentially. On the other hand, it was shown in \cite{BT10}  that if $M(s)=K(s)=\smash{s^\alpha}$, $s\ge1$, for some $\alpha>0$ and if $X$ is a Hilbert space then \eqref{eq:MK_rate} may be replaced by the optimal estimate $\smash{\|T(t)A\inv\|=O(t^{-1/\alpha})}$, $t\to\infty$. This Hilbert space result has subsequently been extended, first in \cite{BaChTo16} and then rather substantially in \cite{RoSeSt17}. On the other hand, it was also shown in \cite{BT10} that in the above polynomial case the upper bound in \eqref{eq:MK_rate} is sharp if we impose no restrictions on the Banach space $X$; see also \cite{BaBoTo16, Sta18}. These arguments show optimality of the quantified Ingham-Karamata theorem both for operator semigroups and for the more general case of functions whose Laplace transforms satisfy suitable conditions. At the heart of these proofs lies an extremely delicate construction of a certain complex measure. In our recent paper \cite{DebSei18} we presented a significantly simpler optimality proof based on a striking application of the open mapping theorem. The results in \cite{DebSei18} go beyond the polynomial case discussed above, and this generalisation had already been achieved in \cite{Sta18} by extending the existing more complicated technique. However, in the semigroup case both the results in \cite{DebSei18} and those in \cite{Sta18} essentially require the functions $M$ and $K$ to grow at least polynomially, and it was left open whether Theorem~\ref{thm:MK} is optimal for more slowly growing functions.

The aim of this paper, which can be viewed as a follow-up contribution to \cite{DebSei18}, is to prove a more general optimality result in the semigroup setting which in particular imposes no lower bound whatsoever on the functions $M$ and $K$ appearing in Theorem~\ref{thm:MK}. Furthermore, we  improve on the value of the constant $c$ appearing in our earlier optimality results \cite[Theorems~2.2 and 2.4]{DebSei18}, which is significant now that $M_{K}$ is allowed to have subpolynomial growth. In fact, we shall obtain the value $c=1$, which is best possible in view of the fact that \eqref{eq:MK_rate} holds  for all $c\in(0,1)$ when $M=K$. We achieve our results by refining the technique used in \cite{DebSei18}. In particular, we first construct, in Lemma~\ref{lem:sequ} below,  a sequence of functions with certain properties, which we then use to prove an important preliminary result, Theorem~\ref{thm:opt}, which can be viewed as proving optimality of a particular variant of the quantified Ingham-Karamata theorem for scalar-valued functions and is of considerable intrinsic interest. These preliminary results are presented in Section~\ref{sec:prelim}. Then in Section~\ref{sec:sg} we return to a construction originally given in \cite{BT10} which allows us to prove, in Theorem~\ref{thm:sg}, that the upper bound in \eqref{eq:MK_rate} is sharp even for slowly growing functions $M$ and $K$.
                         
We use standard notation throughout. In particular, we let $\RR_+=[0,\infty)$ and $\NN=\{1, 2, 3,\dots\}$. For real-valued quantities $x, y$ we write $x\lesssim y$ if there exists a constant $C>0$ such that $x\le Cy$, and we furthermore make use of standard asymptotic notation, such as `big O' and `little o'. For background material on the theory of $C_0$-semigroups we refer the reader to \cite{ABHN11}.

\section{Preliminary results}\label{sec:prelim}

We begin with a technical lemma which is central to this paper.

\begin{lem} \label{lem:sequ}There exist complex-valued functions $f_{k}\in W^{1,\infty}(\RR)\cap W^{1,1}(\RR)$, $k\in\NN$, such that $f_k'$ is uniformly continuous for each $k\in\NN$ and the following properties hold:
\begin{enumerate}
 \item[\textup{1.}]
\begin{enumerate}
 \item[\textup{(a)}] The functions $f_k$, $k\in\NN$, are uniformly bounded in $W^{1,\infty}(\RR)$ and in $W^{1,1}(\RR)$.
 \item[\textup{(b)}] We have $\inf_{k\in\NN}|f_k(0)|>0$.
 \end{enumerate}
 \item[\textup{2.}]
\begin{enumerate}
 \item[\textup{(a)}] There exist constants $\ep_0\in(0,1)$ and $C>0$ such that the functions
 \begin{equation}\label{eq:transform}
\widehat{f_k}(\lambda)=\int_\RR e^{-\lambda t}f_k(t)\,\dd t,\quad k\in\NN,\ \lambda\in i\RR,
\end{equation}
extend analytically to  $\{\lambda\in\CC:|\lambda|<\ep_0\}$ and satisfy $|\widehat{f_k}(\lambda)|\le C|\lambda|^k,$ for all $k\in\NN$ and $|\lambda|<\ep_0.$
 \item[\textup{(b)}] There exists a constant $c>0$ such that for each $k\in\NN$ the function $\widehat{f_k}$ defined in \eqref{eq:transform} extends analytically to the strip 
$$S_{k,c}=\left\{\lambda\in\CC:|{\R\lambda}|< \frac{1}{c\log(k+1)}\right\},$$  
and moreover
 $\sup_{k\in\NN}\sup_{\lambda\in S_{k,c}}|\lambda\widehat{f_k}(\lambda)|<\infty.$
 \end{enumerate}
\end{enumerate}
Furthermore, the constant $c>0$ in \textup{2.(b)} can be chosen arbitrarily small. 
\end{lem}
\begin{proof} 
Note first that by considering a shifted sequence it suffices to define the functions $f_k\colon\RR\to\CC$ only for $k\ge k_0$ where $k_0\in\NN$ may be large. In fact, we shall initially define $f_k$ only for integers $k\in\NN$ which are sufficiently large and of a rather particular form. More specifically, with each even integer $m\in\NN$ we associate the natural number $k_m=m2^m$, and we shall construct a suitable function $f_{k_m}$ for all sufficiently large $m$. To obtain the full sequence, we may then take $f_{k}$ to be the function $\smash{f_{k_{m(k)}}}$, where $m(k)$ is the unique even integer such that $k_{m(k) -2} < k \leq k_{m(k)}$ when $k$ is sufficiently large.  Note that  
$$ \frac{k_{m(k)}}{k} < \frac{k_{m(k)}}{k_{m(k) - 2}} = \frac{4 m(k)}{m(k) - 2} \leq 5$$
when $k\in\NN$ is sufficiently large, and hence $k\le \smash{k_{m(k)}}\le 5k$. In particular, it suffices to construct the subsequence $f_{k_{m}}$ with the desired properties, as the induced sequence $f_{k}$ will inherit them. For simplicity of notation we shall write $f_m$ instead of $f_{k_m}$ in what follows. Given an even integer $m\in\NN$ let
\begin{equation}\label{eq:f}
f_m(t)=\frac{1}{2\pi}\int_\RR e^{its} \frac{F(s)G_{m}(s)H_{m}(s)}{1+s^4}\,\dd s,\quad t\in\RR,
\end{equation}
where 
$$ F(s) = \frac{(s^{2} - 25^2)^{4}}{(1+s^{2})^{4}},\ \;G_m(s) =  \frac{s^{m}}{25^{m} + s^{m}},\ \; H_m(s) = \frac{s^{m2^m}}{(1+s^{m})^{2^m}},\quad s\in\RR.$$
Thus $f_m$ is defined in terms of its Fourier transform, and it follows that
\begin{equation}\label{eq:Fourier}
\widehat{f_m}(is)=\frac{F(s)G_{m}(s)H_{m}(s)}{1+s^4},\quad s\in\RR.
\end{equation}
It is immediately clear that $f_m\in W^{1,\infty}(\RR)\cap W^{1,1}(\RR)$ and that $f_m'$ is uniformly continuous for each $m$. We verify the remaining properties in turn.

In order to verify condition 1.(a) we show explicitly that the sequence $f_m$ is uniformly bounded in $L^\infty(\RR)$ and in $L^1(\RR)$. Once this has been established, the same argument applied to the derivatives
$$f'_m(t)=\frac{1}{2\pi}\int_\RR is e^{its} \frac{F(s)G_{m}(s)H_{m}(s)}{1+s^4}\,\dd s,\quad t\in\RR,$$ 
will yield the result. In fact, it is straightforward to see from the definitions of $F,\ G_m$ and $H_m$ that the functions $f_m$ are uniformly bounded in $L^\infty(\RR)$, so we focus on the $L^1$-estimates. Since the function $s\mapsto(1+s^4)\inv$ is an element of $W^{2,1}(\RR)$, integration by parts shows that $f_m(t)\lesssim C_m(1+t^2)\inv$, $t\in\RR$, where
\begin{equation}\label{eq:sup}
C_m= \sup\left\{ \max_{0\le j_1+j_2+j_3\le2}|F^{(j_{1})}(s)|  |G_{m}^{(j_{2})}(s)| |H_m^{(j_{3})}(s)|:s\in\RR\right\}.
\end{equation}
The result will follow once we have proved that the numbers $C_m$ are uniformly bounded for all sufficiently large even integers $m$. In order to estimate the supremum we split the real line into several subintervals. First, for $|s| \leq 5$ simple calculations show that $|\smash{G_m^{(j)}(s)}|\lesssim m^j5^{-m}$ and $\smash{|H_m^{(j)}(s)|}\lesssim (m2^m)^j$, where $0\le j\le2$. Since $F\in W^{2,\infty}(\RR)$ it follows easily that the supremum over $|s|\le5$ is bounded uniformly in $m$, as required. For $|s| \geq 5$, we have $|\smash{H_m^{(j)}(s)}|\lesssim 1$ for $0\le j\le2$, so we may concentrate on the estimates for $G_m$ when $|s| \geq 5$. Note that
\begin{equation*}\label{eq:deriv}
G_m'(s)=\frac{m25^{m}s^{m-1}}{(25^{m}+ s^{m})^{2}},\quad s\in\RR.
\end{equation*}
Thus for $5\le|s|\le25(1-m^{-1/2})$ we have
$$|G_m'(s)|\lesssim m\left(1-\frac{1}{m^{1/2}}\right)^m\lesssim m\left(2e\inv\right)^{m^{1/2}}\lesssim1,$$
uniformly in $m$. Similarly, $|G_m'(s)|\lesssim1$ for $|s|\ge25(1+m^{-1/2})$, and an analogous argument shows that $|G_m''(s)|\lesssim1$ in both ranges, with implicit constants which are independent of $m$. It remains to consider the range $25(1-m^{-1/2})\le|s|\le25(1+m^{-1/2})$. Straightforward estimates show that $\|\smash{G_m^{(j)}}\|_{L^\infty}\lesssim m^j$ for $0\le j\le2$, and we also observe that $|F(s)|\lesssim m^{-2}$ and $|F'(s)|\lesssim \smash{m^{-3/2}}$ over the range in question. Together with the fact that $F\in W^{2,\infty}(\RR)$ the above estimates show that in \eqref{eq:sup} the supremum over the range $|s|\ge5$ is also uniformly bounded in $m$, so property 1.(a) holds.

In order to verify property 1.(b) we begin by observing that for $s\in\RR$ with $|s| > 2$ we have $H_m(s)\to1$ as $m\to\infty$. Moreover, the function $H_m$ is even and increasing on the positive half-axis. It follows in particular that $H_m(s)\gtrsim1$ uniformly in $m$ and $s$ for $|s|\ge25$. Since $G_m(s)\ge 1/2$ for $|s|\ge25$, we deduce from \eqref{eq:f} that 
$$f_m(0)\gtrsim\frac{1}{2\pi}\int_{|s|\ge25}\frac{F(s)}{1+s^4}\,\dd s>0,$$
where the implicit constant is independent of $m$. Hence property 1.(b) holds.

We now turn to property 2.(a) and set $\ep_0=1/3$. By \eqref{eq:Fourier} the functions $\smash{\widehat{f_m}}$ extend analytically to  $\{\lambda\in\CC:|\lambda|<\ep_0\}$, and crude estimates yield $|\smash{\widehat{f_m}(\lambda)}|\lesssim|\lambda|^{m2^m}$ for $|\lambda|<\ep_0$. Recalling that $k_m=m2^m$ it follows that 2.(a) is satisfied.

In order to verify property 2.(b) we begin by observing that the function $\lambda\mapsto \lambda F(\lambda)(1+\lambda^4)^{-1}$ is bounded on the (rotated) strip $\{\lambda\in\CC: |{\I \lambda}| < 1/2 \}$, so by \eqref{eq:Fourier} and the fact that $\log(k_m+1)\gtrsim m$ it suffices to show that the functions $G_m$ and $H_m$ are uniformly bounded in the strips $S_m=\{\lambda\in\CC: |{\I \lambda}| < 1/2m \}$. In fact, it is enough to show that the functions $Q_m(\lambda) =\lambda^m(1+\lambda^m)\inv$, $\lambda\in\CC$, satisfy $\sup_{\lambda\in S_m}|Q_m(\lambda)|\le1$ for all sufficiently large even integers $m$. Indeed, the corresponding claim for the functions $H_m$ then follows immediately, and for $G_m$ it is a consequence of the fact that $G_m(\lambda)=Q_m(\lambda/25)$, $\lambda\in\CC$. It is straightforward to see that if $|\lambda|<3/4$ then $|Q_m(\lambda)|\le1$ for all sufficiently large $m$. Suppose therefore that $\lambda\in S_m$ satisfies $|\lambda|\ge3/4$ and suppose for the moment that $\R\lambda>0$. We write $\lambda=re^{i\theta}$ where $r\ge3/4$ and $|\theta|\le\pi/2$. Since $|{\I\lambda}|<1/2m$ we have
$$|\theta| \leq \frac{\pi |\sin \theta|}{2} \leq \frac{\pi}{3m}.$$
Hence the argument of $\lambda^{m}$ lies between $-\pi/3$ and $\pi/3$ and, in particular, $\R\lambda^m>0$.  Thus $|\lambda^{m}| \leq |1+\lambda^{m}|$, as required. A similar argument using the fact that $m$ is even applies if $\R\lambda<0$, so we obtain property 2.(b). 

Finally, if we suppose that the functions $f_k$, $k\in\NN$, have all the required properties and satisfy condition 2.(b) for the value $c=c_0>0$ then for any $\alpha\ge1$ the functions $t\mapsto f_k(\alpha t)$, $t\in\RR$, $k\in\NN$, also have the required properties and satisfy 2.(b) for the value $c=c_0/\alpha$.
\end{proof}

We shall also require the following technical result; see \cite[Lemma~2.1]{DebSei18}. 

\begin{lem}\label{lem}
Let $M,K, r\colon\RR_+\to(0,\infty)$ and assume that $M$ and $K$ are non-decreasing and continuous. Suppose that 
$$|f(t)|=O\big(r(t)\inv\big),\quad t\to\infty,$$
 for every bounded Lipschitz continuous function $f\colon\RR_+\to\CC$ such that $f'$ is uniformly continuous and the Laplace transform of $f$ extends analytically to the region $\Omega_M$ defined in \eqref{eq:Omega}
 and satisfies the bound 
\begin{equation}\label{eq:bound}
 \sup_{\lambda \in \Omega_{M}} \frac{|\lambda \widehat{f}(\lambda)|}{K(|{\I\lambda}|)} < \infty.
\end{equation}
Then also 
$$|g(t)|=O\big(r(|t|)\inv\big),\quad |t|\to\infty,$$
for every bounded Lipschitz continuous function $g\in W^{1,1}(\RR)$ such that $g'$ is uniformly continuous and the function $\widehat{g}$ defined as in \eqref{eq:transform} extends analytically to the region
\begin{equation}\label{eq:Omega'}
\Omega'_M=\left\{\lambda\in\CC:|{\R\lambda}|<\frac{1}{M(|{\I\lambda}|)}\right\}\end{equation}
and satisfies 
\begin{equation}\label{eq:bound'}
\sup_{\lambda\in \Omega'_M}\frac{|\lambda \widehat{g}(\lambda)|}{K(|{\I\lambda}|)}<\infty.
\end{equation}
\end{lem}

The following result is a strengthened version of the second part of \cite[Theorem~2.2]{DebSei18}, where a similar result is proved when $M=K$ but with the additional assumption that $M$ grows at least polynomially. The result is ancillary in nature for the purposes of the present paper, but in fact it shows optimality of a certain quantified Tauberian theorem for scalar-valued functions, as discussed in \cite{DebSei18}.

\begin{thm}\label{thm:opt}
Let $M,K, r\colon\RR_+\to(0,\infty)$ be non-decreasing functions and assume that $M$ and $K$ are continuous. Suppose further that the function $\MK\colon\RR_+\to(0,\infty)$ defined as in Theorem~\ref{thm:MK}, satisfies $\MK(s) =O( e^{\alpha s})$ as $s\to\infty$ for some $\alpha>0$. If 
\begin{equation}\label{eq:f_r2}
|f(t)|=O\big(r(t)\inv\big),\quad t\to\infty,
\end{equation}
 for every bounded Lipschitz continuous function $f\colon\RR_+\to\CC$ such that $f'$ is uniformly continuous and the Laplace transform of $f$ extends analytically to the region $\Omega_M$ defined in \eqref{eq:Omega} and satisfies the bound \eqref{eq:bound}, then
 \begin{equation}\label{eq:result}
 r(t)=O\big(\MK\inv(t)\big),\quad t\to\infty.
 \end{equation}
 \end{thm}

\begin{proof} The proof is similar to that of \cite[Theorem~2.2]{DebSei18}, except that we use the sequence constructed in Lemma~\ref{lem:sequ} at a crucial stage.  Let $X$ be the vector space of all  bounded functions $g \in W^{1,\infty}(\RR)\cap W^{1,1}(\RR)$ such that $g'$ is uniformly continuous and the function $\widehat{g}$ defined in \eqref{eq:transform} extends analytically to the region $\Omega'_M$ defined in \eqref{eq:Omega'} and satisfies the bound  \eqref{eq:bound'}. We endow $X$ with the complete norm
\begin{equation} \label{eq:defnormx}
\|g\|_X=\|g\|_{W^{1,1}}+\|g\|_{W^{1,\infty}}+\sup_{\lambda\in \Omega'_M}\frac{|{\lambda\widehat{g}(\lambda)}|}{K(|{\I\lambda}|)},\quad g\in X.
\end{equation}
Let $Y$ be the set of all functions $g\in X$ such that $|g(t)|=O(r(|t|)\inv)$ as $|t|\to\infty$, endowed with the complete norm
$$\|g\|_Y=\|g\|_X+\sup_{t\in\RR}|g(t)|r(|t|),\quad g\in Y.$$
It follows from the definitions that $Y$ is continuously embedded in $X$, and by our assumptions and Lemma~\ref{lem} we also have $X\subseteq Y$. By the open mapping theorem the embedding of $Y$ into $X$ is an open map, and hence $\|g\|_Y\lesssim\|g\|_X$, $g\in X$, and in particular
\begin{equation}\label{eq:g_bd}
\sup_{t\in\RR}|g(t)|r(|t|)\lesssim \|g\|_X,\quad g\in X.
\end{equation}
We now consider specifically chosen functions $g\in X$. Indeed, consider the functions $g_{k,R,t}$ defined for $k\in\NN$, $R>0$ and $t\in\RR$ by $g_{k,R,t}(s)=f_k(R(s-t))$, $s\in\RR$, where the functions $f_k\in W^{1,\infty}(\RR)\cap W^{1,1}(\RR)$, $k\in\NN$, are as in Lemma~\ref{lem:sequ}.  In particular, $f_k'$ is uniformly continuous and
$$\widehat{g_{k,R,t}}(\lambda)=\frac{e^{-t\lambda}}{R}\widehat{f_k}\left(\frac{\lambda}{R}\right),\quad \lambda\in i\RR,$$
so  provided that
\begin{equation}\label{eq:cond}
R \geq \frac{c \log (k+1)}{M(0)}
\end{equation} 
we have $g_{k,R,t}\in X$ by property 2.(b) of Lemma~\ref{lem:sequ}. By properties 1.(a) and 1.(b) of Lemma~\ref{lem:sequ} we obtain 
\begin{equation} \label{eq:est}
 r(t) \lesssim \sup_{s \in \RR} |g_{k,R,t}(s)|r(|s|) \lesssim R +  R\inv\sup_{\lambda \in \Omega'_M} \left|\frac{\lambda e^{-t\lambda}}{K(|{\I \lambda}|)}\widehat{f_k}\left(\frac{\lambda}{R}\right)\right|.
\end{equation}
Here the implicit constants are independent of $k, R$ and $t$. We shall choose $R$ and $k$ depending on $t$ so as to obtain a particularly sharp upper bound in \eqref{eq:est}. To estimate the supremum suppose first that $|{\I \lambda}| > \varepsilon R$, where the value of $\varepsilon\in(0,1)$ is sufficiently small in a sense to be made precise below. Since $M$ and $K$ are non-decreasing we see using \eqref{eq:cond} and property 2.(b) of Lemma~\ref{lem:sequ} that
\begin{equation} \label{eq:est1}
 \left|\frac{\lambda e^{-t\lambda}}{K(|{\I \lambda}|)}\widehat{f_k}\left(\frac{\lambda}{R}\right)\right| \lesssim \frac{R}{K(\varepsilon R)}\exp\left(\frac{t}{M(\varepsilon R)}\right)
 \end{equation}
for all $\lambda\in\Omega_M'$ such that $|{\I \lambda}| > \varepsilon R$. Now let $\lambda\in\Omega_M'$ with $|{\I \lambda}| \le \varepsilon R$. By property 2.(a) of Lemma~\ref{lem:sequ}, if $R\ge(\ep M(0))\inv$ and $\ep\in(0,\ep_0/2)$ then 
\begin{equation}\label{eq:est2}
\left|\frac{\lambda e^{-t\lambda}}{K(|{\I \lambda}|)}\widehat{f_k}\left(\frac{\lambda}{R}\right)\right| \lesssim R\exp\left(\frac{t}{M(0)}\right)(2\ep)^{k}.
\end{equation}
We now choose $k=\lfloor t\rfloor$ for $t\ge1$ and assume in addition that $2\ep\le\exp(-M(0)\inv)$. Then \eqref{eq:est2} becomes
\begin{equation} \label{eq:est3}
\left|\frac{\lambda e^{-t\lambda}}{K(|{\I \lambda}|)}\widehat{f_k}\left(\frac{\lambda}{R}\right)\right| \lesssim R.
\end{equation}
We now set $R=\ep\inv \smash{\MK\inv(t)}$ for $t\ge1$ sufficiently large. Let $C>0$ be such that $\MK(s)\le Ce^{\alpha s}$, $s\ge0$, and observe that for sufficiently large values of $t$ we have
$$R\ge\frac{\log (C\inv t)}{\ep\alpha}\ge \frac{\log(k+1)}{\alpha},$$
so that \eqref{eq:cond} holds provided the constant $c>0$ appearing in 2.(b) of Lemma~\ref{lem:sequ} satisfies $c\le\alpha\inv M(0)$, as we may assume it does by the final statement in Lemma~\ref{lem:sequ}. It follows from \eqref{eq:est}, \eqref{eq:est1} and \eqref{eq:est3} that $r(t)=O(\MK\inv(t))$ as $t\to\infty$, so the proof is complete.
\end{proof}

\section{Optimal decay for operator semigroups}\label{sec:sg}

We now come to our main result, which proves that Theorem~\ref{thm:MK} is optimal in a strong sense. Our proof closely follows that of \cite[Theorem~3.1]{DebSei18}, which in turn is based on \cite[Theorem~7.1]{BaBoTo16}, \cite[Theorem~4.1]{BT10} and \cite[Theorem~4.10]{Sta18}. We shall say that a function $M\colon\RR_+\to(0,\infty)$ is \emph{regularly growing} if it is  non-decreasing and continuous and there exists $c\in(0,1)$ such that 
\begin{equation}\label{eq:reg}
M(s)\ge c M\left(s+\frac{c}{M(s)}\right),\quad s\ge0.
\end{equation}
As discussed in \cite{DebSei18} this is a very mild regularity condition, and in particular the condition is significantly milder than those required in \cite[Section~4]{Sta18}. Note also that if a function $M$ is regularly growing then \eqref{eq:reg} is necessarily satisfied for all sufficiently small $c\in(0,1)$. We emphasise, however, that unlike in \cite{DebSei18} our definition of regular growth no longer includes any growth conditions. We shall still require an upper bound in our main result below, although as discussed in Section~\ref{sec:intro} at least in the special case where $M=K$ this entails no significant loss of generality. Crucially, though, we no longer impose any lower bounds on the functions $M$ and $K$.

\begin{thm}\label{thm:sg}
Let $M,K\colon\RR_+\to(0,\infty)$ be regularly growing functions and suppose that the function $\MK\colon\RR_+\to(0,\infty)$ defined in Theorem~\ref{thm:MK} satisfies $\MK(s)=O(e^{\alpha s})$ as $s\to\infty$ for some $\alpha>0$. Then there exists a complex Banach space $X$ and a bounded $C_0$-semigroup $\T$ on $X$ with generator $A$ such that $\Omega_{\delta M}\subset\rho(A)$ and
\eqref{eq:K_bound} holds for some $\delta>1$, and moreover
\begin{equation*}\label{eq:lb}
\limsup_{t\to\infty}\big\|\MK\inv(t)T(t)A\inv\big\|>0.
\end{equation*}
\end{thm}

\begin{proof}
Let $X$ be the vector space of all bounded uniformly continuous functions $f\colon\RR_+\to\CC$ whose Laplace transform extends to the region $\Omega=\{\lambda\in\CC:\R\lambda>-M(|\I\lambda|)\inv\mbox{ and }|{\R\lambda}|<1\}$ and satisfies
$$\sup_{\lambda\in\Omega}\frac{|\widehat{f}(\lambda)|}{K(|{\I\lambda}|)}<\infty,$$
endowed with the complete norm
$$\|f\|_X=\|f\|_{L^\infty}+\sup_{\lambda\in\Omega}\frac{|\widehat{f}(\lambda)|}{K(|{\I\lambda}|)},\quad f\in X.$$
Arguing as in the proof of  \cite[Theorem~7.1]{BaBoTo16} and \cite[Theorem~4.1]{BT10} we see that the left-shift semigroup $\T$ is a well-defined bounded $C_0$-semigroup on $X$ whose generator $A$, the differentiation operator on an appropriate domain, has all the required properties. Note that condition \eqref{eq:reg} is chosen precisely in such a way that all the arguments extend without major adjustments to our more general setting. Suppose for the sake of  contradiction that $\|T(t)A\inv\|=o(\smash{\MK\inv}(t)\inv)$ as $t\to\infty$. Then we may find a non-decreasing function $r\colon\RR_+\to(0,\infty)$ such that $\|T(t)A\inv\|=O(r(t)\inv)$ and $\smash{\MK\inv(t)}=o(r(t))$ as $t\to\infty.$ If $f\colon\RR_+\to\CC$ is a bounded Lipschitz continuous function such that $f'$ is uniformly continuous and the Laplace transform of $f$ extends analytically to the region $\Omega_M$ defined in \eqref{eq:Omega} and satisfies the bound \eqref{eq:bound}, then $f,f'\in X$ and $f=A\inv f'$. Hence  
$$|f(t)|\le\|T(t)f\|_{L^\infty}\le\|T(t)A\inv f'\|_X=O\big(r(t)\inv\big),\quad t\to\infty.$$
Hence $r(t)=O(\smash{\MK\inv(t))}$  as $t\to\infty$ by Theorem~\ref{thm:opt}, a contradiction.
\end{proof}

\begin{rem}
It is possible to weaken the regularity conditions in Theorem~\ref{thm:sg} slightly. For instance, instead of requiring that $M$ is regularly growing it would be sufficient to assume that
$$M(s)\ge c M\left(s+\frac{c}{K(s)}\right),\quad s\ge0,$$
for some, and hence all sufficiently small, $c\in(0,1)$. This assumption is marginally weaker since both $M$ and $K$ are non-decreasing and $K(s)\ge M(s),$ $s\ge0$. Moreover, by looking carefully at the details of the proofs of  \cite[Theorem~7.1]{BaBoTo16} and \cite[Theorem~4.1]{BT10} one sees that the semigroup $\T$ in Theorem~\ref{thm:sg} actually satisfies the conditions of the statement for many or indeed all values of  $\delta>1$ provided that $M(s)$ has no sudden growth spurts for sufficiently large values of  $s\ge0$.
\end{rem}

\end{document}